\documentclass[regno]{amsart}

\usepackage{amsmath}
\usepackage{amsthm}
\usepackage{amssymb}
\usepackage{bbm}
\usepackage{graphicx}
\usepackage{graphicx}
\usepackage{wrapfig}
\usepackage{float}
\usepackage[font=small]{caption}
\usepackage[usenames,dvipsnames]{xcolor}
\usepackage{hyperref}
\usepackage{stmaryrd}
\usepackage{amscd}
\usepackage{amssymb}
\usepackage{youngtab}
\usepackage{pdfpages}
\usepackage[all,cmtip]{xy}
\usepackage{animate}
\usepackage{tikz}

\newtheorem{theorem}{Theorem}[section]
\newtheorem{lemma}[theorem]{Lemma}
\newtheorem{proposition}[theorem]{Proposition}
\newtheorem{corollary}[theorem]{Corollary}
\newtheorem{definition}[theorem]{Definition}
\newtheorem{remark}[theorem]{Remark}

\DeclareMathOperator{\Out}{Out}
\DeclareMathOperator{\Aut}{Aut}

\DeclareMathOperator{\Hom}{Hom}

\DeclareMathOperator{\Ind}{Ind}

\DeclareMathOperator{\im}{im}

\newcommand{\FIMod}{\mathsf{FI\mbox{-}Mod}}

\newcommand{\FBMod}{\mathsf{FB\mbox{-}Mod}}
\newcommand{\Vect}{\mathsf{Vect}}
\newcommand{\FI}{\mathsf{FI}}

\newcommand{\FB}{\mathsf{FB}}

\title{On the FI-module structure of $H^i(\Gamma_{n,s})$}
\author{Amin Saied}
\address{Department of Mathematics\\ Cornell University\\ 120 Malott Hall\\ Ithaca\\ NY\\ 14850}
\email{as2789@cornell.edu}

\begin{document}
\maketitle
\begin{abstract}
The groups $\Gamma_{n,s}$ are defined in terms of homotopy equivalences of certain graphs, and are natural generalisations of $\Out(F_n)$ and $\Aut(F_n)$. They have appeared frequently in the study of free group automorphisms, for example in proofs of homological stability in \cite{HV1, HV2} and in the proof that $\Out(F_n)$ is a virtual duality group in \cite{BestvinaFeighn1}. More recently, in \cite{CHKV1}, their cohomology $H^i(\Gamma_{n,s})$, over a field of characteristic zero, was computed in ranks $n=1, 2$ giving new constructions of unstable homology classes of $\Out(F_n)$ and $\Aut(F_n)$. In this paper we show that, for fixed $i$ and $n$, this cohomology $H^i(\Gamma_{n,s})$ forms a finitely generated FI-module of stability degree $n$ and weight $i$, as defined by Church-Ellenberg-Farb in \cite{CEF}. We thus recover that for all $i$ and $n$, the sequences $\{H^i(\Gamma_{n,s})\}_{s\geq0}$ satisfy representation stability, but with an improved stable range of $s \geq i+n$ which agrees with the low dimensional calculations made in \cite{CHKV1}. Another important consequence of this FI-module structure is the existence of character polynomials which determine the character of the $\mathfrak{S}_s$-module $H^i(\Gamma_{n,s})$ for all $s \geq i+n$. In particular this implies that, for fixed $i$ and $n$, the dimension of $H^i(\Gamma_{n,s})$, is given by a polynomial in $s$ for all $s\geq i+n$. We compute explicit examples of such character polynomials to demonstrate this phenomenon.
\end{abstract}

\section{Introduction}

It is well known that the group of outer automorphisms of the free group of rank $n$ can be described as the space of self-homotopy equivalences of a graph $X_n$ of rank $n$
, up to homotopy, i.e., 
\begin{equation}
\Out(F_n) \cong \pi_0(HE(X_n)). \notag
\end{equation}
Similarly the full group of automorphisms of the free group of rank $n$ is the space of homotopy equivalences of a graph $X_{n,1}$ of rank $n$ with a distinguished basepoint $\partial$, up to homotopy,
\begin{equation}
\Aut(F_n) \cong \pi_0(HE(X_{n,1})), \notag
\end{equation}
where homotopies are required to fix the basepoint throughout.\\

\begin{figure}[h]
\begin{center}
\begin{tikzpicture}[scale=1]
\draw (0,0) circle (1cm);
\draw[rounded corners = 30pt] (0,-1) -- (-0.7,0) -- (0,1);
\draw[rounded corners = 30pt] (0,-1) -- (0.7,0) -- (0,1);
\node at (0,-1.7) {$X_3$};

\begin{scope}[xshift = 5cm]
\draw (0,0) circle (1cm);
\draw[rounded corners = 30pt] (0,-1) -- (-0.7,0) -- (0,1);
\draw[rounded corners = 30pt] (0,-1) -- (0.7,0) -- (0,1);
\node at (0,-1.7) {$X_{3,1}$};
\draw (320:1) -- (320:1.5) node[above right] {$\partial$};
\node at (320:1.5) {$\bullet$};
\end{scope}
\end{tikzpicture}
\end{center}
\caption{Examples of rank 3 graphs that can be used to define $\Out(F_3)$ and $\Aut(F_3)$.}
\end{figure}
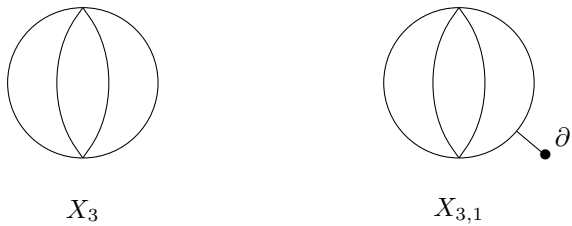

There is a natural generalisation then, where we let $X_{n,s}$ be a graph, by which we mean a connected finite 1-dimensional CW-complex, of rank $n$ with $s$ marked points $\partial=\{x_1, \cdots, x_s\}$. We should then consider the group of self-homotopy equivalences of $X_{n,s}$ fixing $\partial$ pointwise, modulo homotopies through such maps, i.e.,
\begin{equation}
\Gamma_{n,s} := \pi_0(HE(X_{n,s})). \notag
\end{equation}

In this paper we study the structure of the cohomology $H^i(\Gamma_{n,s})$, always over a field of characteristic zero, as a sequence of $\mathfrak{S}_s$-modules. The symmetric group $\mathfrak{S}_s$ acts on $H^i(\Gamma_{n,s})$ as follows. A homotopy equivalence $h:X_{n,s} \rightarrow X_{n,s}$ permuting $\partial$ induces an automorphism of $\Gamma_{n,s}$ by conjugation. This automorphism depends, \emph{a priori}, on the choice of $h$, however, on the level of cohomology it depends only on the permutation. Indeed, if $h$ fixes $\partial$ pointwise then the induced automorphism is inner, and thus induces the identity on cohomology.\\

The groups $\Gamma_{n,s}$ have been used, for example, to show that $\Out(F_n)$ and $\Aut(F_n)$ satisfy homological stability in \cite{HV1, HV2}, and they appeared in \cite{BestvinaFeighn1} in the proof that $\Out(F_n)$ is a virtual duality group. More recently they were used in \cite{CHKV1} to investigate the so called unstable cohomology of $\Out(F_n)$ and $\Aut(F_n)$ by means of an `assembly map'
\begin{equation}
H^i(\Gamma_{n_1, s_1}) \otimes \cdots \otimes H^i(\Gamma_{n_k, s_k}) \rightarrow H^i(\Gamma_{n,s}). \notag
\end{equation}

In particular, in \cite{CHKV1} they compute $H^i(\Gamma_{n,s})$ as an $\mathfrak{S}_s$-module for rank $n=1,2$ and use these computations to assemble homology classes in the unstable range of $\Out(F_n)$ and $\Aut(F_n)$. Moreover these computations show that, in rank $n=1,2$ and for fixed $i\geq 0$ the sequence $\{H^i(\Gamma_{n,s})\}_{s\geq0}$ satisfies representation stability; a representation theoretic analogue of homological stability defined by Thomas Church and Benson Farb in \cite{CF} (see Definition \ref{defrepstab}).
In \cite{CHKV1} they use an alternate description of $\Gamma_{n,s}$ as a quotient of a certain mapping class group of a three-manifold, together with general results about representation stability of mapping class groups, to deduce that for any fixed $i$ and $n$ the groups $H^i(\Gamma_{n,s})$ satisfy representation stability with stable range $s\geq 3i$. However, the calculations made in \cite{CHKV1} in rank $n=1,2$ actually adhere to a bound of $s \geq i+n$. In this paper we improve the stable range to agree with these low rank calculations.\\

\textbf{Theorem A.} \emph{For fixed $i$ and $n$ the sequence $H^i(\Gamma_{n,s})$ is uniformly representation stable as $s \rightarrow \infty$ with stable range $s \geq n+i$.}\\

We show this by exhibiting that $H^i(\Gamma_{n,s})$ defines an FI-module (see Definition \ref{defFI}). Building on their work in \cite{CF}, and together with Jordan Ellenberg and Rohit Nagpal, the theory of FI-modules was developed \cite{CEF, CEFN}, facilitating the application of homological techniques to sequences of $\mathfrak{S}_s$-modules. We use these techniques to prove the following theorem.\\

\textbf{Theorem B.} \emph{$H^i(\Gamma_{n,\bullet})$ is a finitely generated FI-module of stability degree $n$ and weight $i$.}\\

An important feature of finitely generated FI-modules is the existence of character polynomials; integer-valued polynomials in $\mathbb{Q}[X_1, X_2, \ldots]$ where $X_i:\mathfrak{S}_s \rightarrow \mathbb{N}$ is the class function that counts the number of $i$-cycles. Let $\chi_{H^i(\Gamma_{n,s})}$ denote the character of the $\mathfrak{S}_s$-module $H^i(\Gamma_{n,s})$. 

\begin{corollary}\label{char} There exists a character polynomial $f \in \mathbb{Q}[X_1, \ldots, X_i]$ depending on $i$ and $n$ such that for all $s \geq i + n$ and all $\sigma \in \mathfrak{S}_s$,
$$\chi_{H^i(\Gamma_{n,s})}(\sigma) = f(\sigma).$$
In particular, the dimension of $H^i(\Gamma_{n,s})$ is given by the polynomial $f(s,0, \ldots, 0)$.
\end{corollary}

One consequence of this result is that, for $s$ sufficiently large, the character $\chi_{i,n}$ is insensitive to cycles of length greater than $i$. We highlight this phenomenon by computing examples of these stable character polynomials in Section \ref{subsec_character_poly}.\\

Theorem A and Corollary \ref{char} follow immediately from Theorem B by results of Church-Ellenberg-Farb in \cite{CEF} that we state below (Proposition \ref{CEF255}, Proposition \ref{CEF334}). In Section \ref{SecFI} we recall the basic definitions and properties of FI-modules which are then used in Section \ref{Sec:Groups} to show that $H^i(\Gamma_{n,s})$ forms an FI-module. In Section \ref{SecSSA} we use a Leray-Serre spectral sequence to prove Theorem B.\\

\subsection*{Acknowledgements.} The author thanks his thesis advisor Martin Kassabov for his insight and for many helpful conversations.

\section{Representation Stability and FI-modules} \label{SecFI}
We work throughout over a field $\mathbbm{k}$ of characteristic 0. Let $\mathfrak{S}_s$ denote the symmetric group on $s$ letters. Recall that the irreducible representations of $\mathfrak{S}_s$ correspond to partitions $\lambda$ of $s$, denoted $\lambda \vdash s$, where $\lambda = (\lambda_1, \lambda_2, \ldots, \lambda_k)$ such that $\lambda_1\geq \lambda_2\geq \cdots \geq \lambda_k$ and $\sum_i \lambda_i = s =: |\lambda|$. We denote the irreducible representation corresponding to $\lambda \vdash s$ by $P_\lambda$. Let $\lambda \vdash q$. It will be useful to denote by $\lambda[s]$ the `padded' partition $(s-|\lambda|, \lambda_1, \lambda_2, \cdots)$ and by $P_{\lambda[s]}=P(\lambda)_s$ its corresponding irreducible representation.\\

Let $P$ be a $\mathfrak{S}_a$-module, and $Q$ an $\mathfrak{S}_b$-module, then we denote the induced representation by
\begin{equation}
P \circ Q = \Ind_{\mathfrak{S}_a\times \mathfrak{S}_b}^{\mathfrak{S}_{a+b}} P \otimes Q \notag
\end{equation}
(see, for example, Fulton-Harris \cite{FH}).\\

Also, following \cite{CHKV1}, we denote by $V^{\wedge k}$ the $\mathfrak{S}_k$-module which is isomorphic as a vector space to $V^{\otimes k}$ where $\mathfrak{S}_k$ acts by permuting the factors and multiplying by the sign of the permutation. That is, let $\mbox{alt}=P_{(1^k)}$ denote the alternating representation, then
\begin{equation}
V^{\wedge k} = V^{\otimes k} \otimes \mbox{alt}. \notag
\end{equation}

For convenience, we recall from \cite{CEF} the definitions of representation stability and of FI-modules.

\begin{definition} A sequence $\{V_s\}$ of finite dimensional $\mathfrak{S}_s$-modules together with linear maps $\phi_s:V_s \rightarrow V_{s+1}$ is called \textbf{consistent} if for each $s$ and each $g \in \mathfrak{S}_s$ the following diagram commutes:
\begin{equation}
\xymatrix{ V_s \ar[r]^{\phi_s} \ar[d]_g & V_{s+1} \ar[d]^{i(g)} \\
V_s \ar[r]_{\phi_s} & V_{s+1}
} \notag
\end{equation}
where $i : \mathfrak{S}_s \rightarrow \mathfrak{S}_{s+1}$ is the standard inclusion.
\end{definition}

\begin{definition}\label{defrepstab} A consistent sequence $\{V_s\}$ is called \textbf{representation stable} if for sufficiently large $s$ each of the following holds.
\begin{enumerate}
\item \emph{(Injectivity)} The maps $\phi_s:V_s \rightarrow V_{s+1}$ are injective.
\item \emph{(Surjectivity)} The $\mathfrak{S}_{s+1}$-span of $\phi_s(V_s)$ equals all of $V_{s+1}$.
\item \emph{(Multiplicities)} The decomposition of $V_s$ into irredicible $\mathfrak{S}_s$-modules is of the form
\begin{equation}
V_s = \bigoplus_{\lambda} c_{\lambda,s} P(\lambda)_s \notag
\end{equation}
with multiplicities $0 \leq c_{\lambda,s} \leq \infty$. Moreover, for each $\lambda$, the multiplicities $c_{\lambda, s}$ are eventually independent of $s$.
\end{enumerate}
The sequence is called \textbf{uniformly representation stable} with stable range $s \geq S$ if in addition, the multiplicities $c_{\lambda,s}$ are independent of $s$ for all $s \geq S$ with no dependence on $\lambda$.
\end{definition}

Let $\FI$ be the category whose objects are finite sets and whose morphisms are injections. This is equivalent to the category whose objects are natural numbers $\textbf{s}$ and whose morphisms $\textbf{t} \rightarrow \textbf{s}$ correspond to injections $\{1, \ldots, t\} \rightarrow \{1, \ldots, s\}$.

\begin{definition}\label{defFI} An \textbf{FI-module} (over a field $\mathbbm{k}$) is a functor $V$ from $\FI$ to the category $\Vect_{\mathbbm{k}}$ of vector spaces over $\mathbbm{k}$. We denote the vector space $V(\textbf{s})$ by $V_s$ and refer to the maps $V(\phi)$ for $\phi$ a morphism in $\FI$ as the structure maps of $V$.
\end{definition}

\subsection{Stability degree and weight of an FI-module}
We recall the notions of stability degree and weight of an FI-module. These will imply representation stability and can be used to control the stable range of a sequence $\{V_s\}$ corresponding to an FI-module $V$.\\

Let $V_s$ be an $\mathfrak{S}_s$-module. Then we denote by $(V_s)_{\mathfrak{S}_s}$ the $\mathfrak{S}_s$-coinvariant quotient $V_s \otimes_{\mathbbm{k}\mathfrak{S}_s} \mathbbm{k}$.

\begin{definition} An FI-module $V$ has \textbf{stability degree} $t$ if for all $a \geq 0$, the maps $(V_{s+a})_{\mathfrak{S}_s} \rightarrow (V_{s+1+a})_{\mathfrak{S}_{s+1}}$ induced by the standard inclusions $I_s: \{1, \ldots, s\} \rightarrow \{1, \ldots, s+1\}$, are isomorphisms for all $s \geq t$.\\

We say $V$ has \textbf{injectivity degree} $\mathcal{I}$ (resp. \textbf{surjectivity degree} $\mathcal{S}$) if the maps $(V_{s+a})_{\mathfrak{S}_s} \rightarrow (V_{s+1+a})_{\mathfrak{S}_{s+1}}$ are injective $\forall \ s \geq \mathcal{I}$ (resp. surjective $\forall \ s \geq \mathcal{S}$). We say $V$ has \textbf{stability type} $(\mathcal{I, S})$.
\end{definition}

\begin{definition} An FI-module $V$ has \textbf{weight} $\leq d$ if for all $s \geq 0$ every irreducible component $P(\lambda)_s$ of $V_s$ has $|\lambda| \leq d$.
\end{definition}

\begin{remark}
A key property of weight is that it is preserved under subquotients and extensions. In fact, there is an alternate definition of weight: the collection of FI-modules over $\mathbbm{k}$ of weight $\leq d$ is the minimal collection which contains all FI-modules generated in degree $\leq d$ and is closed under subquotients and extensions. For more details see \cite{CEF}.
\end{remark}

Together, finite weight and stability of an FI-module $V$ imply representation stability of the corresponding sequence of $\mathfrak{S}_s$-modules $\{V_s\}$. Moreover, the stability degree and weight give a measure of control on the representation stable range. The following result of Church-Ellenberg-Farb is the key to deducing representation stability of $\{ H^i(\Gamma_{n,s}) \}_{s \geq 0}$ from our FI-module $H^i(\Gamma_{n,\bullet})$.

\begin{proposition}[\cite{CEF}, Proposition 3.3.3] \label{CEF255} Let $V$ be an FI-module of weight $d$ and stability degree $t$. Then the sequence $\{V_s\}$ is uniformly representation stable with stable range $s \geq t+d$.
\end{proposition}

With this in hand it is clear that Theorem A is an immediate corollary to Theorem B. Another consequence of Theorem B is the existence of stable character polynomials.

\subsection{Character Polynomials}\label{subsec_character_poly}
For $j \geq 1$, let $X_j: \mathfrak{S}_s \rightarrow \mathbb{N}$ be the class function defined by
$$X_j(\sigma) = \mbox{number of $j$-cycles in $\sigma$}.$$
A polynomial in the variables $X_j$ is called a character polynomial. We define the degree of a character polynomial by setting $\deg(X_j) = j$. The following theorem of Church-Ellenberg-Farb says that characters of finitely generated FI-modules are eventually described by a single character polynomial, and moreover gives explicit bounds on the degree and the stable range of this polynomial in terms of weight and stability degree of the FI-module.

\begin{theorem}[\cite{CEF}, Theorem 3.3.4] \label{CEF334} Let $V$ be a finitely generated FI-module of weight $\leq d$ and stability degree $\leq t$. There exists a unique polynomial $f_V \in \mathbb{Q}[X_1, \dots, X_d]$ of degree at most $d$ such that for all $n \geq d + t$ and all $\sigma \in \mathfrak{S}_n$,
$$\chi_{V_n}(\sigma) = f_V(\sigma).$$
\end{theorem}

Corollary \ref{char} will thus follow immediately from Theorem $B$. It is worth pointing out that in particular, this shows that the dimension of $H^i(\Gamma_{n,s})$ is eventually (as $s$ grows) given by a single character polynomial.\\

In \cite{CHKV1} Conant-Hatcher-Kassabov-Vogtmann describe the $\mathfrak{S}_s$-module structure of $H^i(\Gamma_{n,s})$ for $n=1,2$, from which one can read off their irreducible $\mathfrak{S}_s$-module decomposition, that is, a decomposition into terms of the form $P(\lambda)_s$. We have the following classical fact, which underpins the theorem above. Fix a  partition $\lambda$. There exists a unique character polynomial $f_\lambda$ such that for any $s \geq |\lambda| + \lambda_1$, the character of the $\mathfrak{S}_s$-module $P(\lambda)_s$ is given by $f_\lambda$. In \cite{GG} they describe an algorithm constructing $f_\lambda$ that we will use in conjunction with calculations from \cite{CHKV1} to compute some explicit examples of character polynomials of various $H^i(\Gamma_{n,s})$. It is cleanest to describe these character polynomials in terms of the notation $(x)_j := x(x-1)\cdots(x-j+1)$.\\

\paragraph{\textbf{Example.}} Fix $n=1$ and $i=2$. From \cite{CHKV1}, Proposition 2.7 we obtain the following decomposition of $H^2(\Gamma_{1,s})$ into irreducible $\mathfrak{S}_s$-modules.
$$H^2(\Gamma_{1,s}) = P\left(\Yvcentermath1\scriptsize~\yng(1,1)~\right)_s.$$
Using the algorithm from \cite{GG} we obtain the character polynomial $f_{\Yvcentermath1\tiny~\yng(1,1)}$ for $P\left(\Yvcentermath1\tiny~\yng(1,1)~\right)$.\\
Corollary \ref{char} implies that, for $s \geq 3$, the character $\chi_{2,1}$ of $H^2(\Gamma_{1,s})$ is given by the polynomial,
$$f_{2,1}(X_1, X_2) = f_{\Yvcentermath1\tiny~\yng(1,1)}(X_1, X_2) =   \frac{1}{2}\cdot(X_1)_2 - (X_1) - (X_2) + 1. $$
We can use this, for example, to obtain that for $s \geq 3$ the dimension of $H^2(\Gamma_{1,s})$ is $$\frac{s(s-1)}{2}-s+1 = {s-1 \choose 2}.$$
Notice that this agrees with the description of $H^2(\Gamma_{1,s}) = \bigwedge^2 \mathbbm{k}^{s-1}$ given in \cite{CHKV1}.\\

\paragraph{\textbf{Example.}} Fix $n=2$ and $i=4$. From \cite{CHKV1}, Theorem 2.10 we obtain the following stable decomposition of $H^4(\Gamma_{2,s})$ into irreducible $\mathfrak{S}_s$-modules. For $s \geq 6$,
$$H^4(\Gamma_{2,s}) = P\left(\Yvcentermath1\scriptsize~\yng(2)~\right)_s \oplus P\left(\Yvcentermath1\scriptsize~\yng(2,1)~\right)_s \oplus P\left(\Yvcentermath1\scriptsize~\yng(2,2)~\right)_s.$$
Using the algorithm from \cite{GG} we obtain the character polynomials $f_{\Yvcentermath1\tiny~\yng(2)}, f_{\Yvcentermath1\tiny~\yng(2,1)}$ and $f_{\Yvcentermath1\tiny~\yng(2,2)}$.\\
Corollary \ref{char} implies that, for $s \geq 6$, the character $\chi_{4,2}$ of $H^4(\Gamma_{2,s})$ is given by the sum of these three character polynomials,
$$f_{4,2}(X_1, X_2, X_3, X_4) =   \frac{1}{12}(X_1)_4 + (X_2)_2 - X_1\cdot X_3. $$
For instance, let $\tau = (1~2)(3~4)(5~6\cdots 100) \in \mathfrak{S}_{100}$. Then $\chi_{4,2}(\tau)=2$.\\

Both the stable decomposition of $H^i(\Gamma_{n,s})$ and the stable character polynomials describing $\chi_{i,n}$ evident in these examples are general features of being a finitely generated FI-module. It thus remains to prove Theorem B, which we will do by analysing a spectral sequence of FI-modules. It turns out that the $E_2$-page of that spectral sequence admits a particularly nice description in terms of \emph{free FI-modules}. We recall their definition now.\\

\subsection{Free FI-modules}
Let $\FIMod$ denote the category whose objects are FI-modules over $\mathbbm{k}$ and whose morphisms are given by natural transformations. This is an abelian category in which operations like kernels, quotients, subobjects, etc., are all defined pointwise.\\

Let $\FB$ denote the category of finite sets with bijections, and similarly, let $\FBMod$ denote the category whose objects are FB-modules, i.e., functors from $\FB$ to $\Vect_{\mathbbm{k}}$, and whose morphisms are natural transformations. An FB-module $W$ determines, and is determined by, a collection of $\mathfrak{S}_n$-modules $W_n$ for each $n \in \mathbb{N}$. In particular, any $\mathfrak{S}_n$-module $W_n$ determines an FB-module. There is a natural forgetful functor
\begin{equation}
\pi: \FIMod \rightarrow \FBMod \notag
\end{equation}
sending an FI-module $V$ to the sequence of $\mathfrak{S}_n$-modules it determines.

\begin{definition}[\cite{CEF}, Definition 2.2.2]
Define the functor
\begin{equation}
M:\FBMod \rightarrow \FIMod \notag
\end{equation}
to be the left adjoint to $\pi$.\\

Explicitly $M(P_\lambda)$ takes a finite set $\mathbf{s}=\{1, \ldots, s\}$ to $P_\lambda \circ P_{(s-|\lambda|)}$. We denote $M(P_\lambda)$ simply by $M(\lambda)$.
\end{definition}

In \cite{CEF} they describe the stability degree and weight of $M(\lambda)$. We stress that we are working over a field $\mathbbm{k}$ of characteristic zero, otherwise we cannot guarantee such strong bounds on surjectivity degree.

\begin{proposition}[see \cite{CEF}]\label{stabM} The FI-module $M(\lambda)$ has stability type $(0, \lambda_1)$ and weight at most $|\lambda|$.
\end{proposition}

\begin{proof}
This follows from three results in \cite{CEF}. The injectivity degree is given in Proposition 3.1.7, the surjectivity degree in Proposition 3.2.6 and the weight in Proposition 3.2.4.
\end{proof}

Our notation for irreducible representation of $\mathfrak{S}_s$ as $P(\lambda)_s$ is suggestive of the structure of an FI-module, and indeed this is the case.

\begin{proposition}[\cite{CEF}, Proposition 3.4.1] \label{Prop:Plambda} For any partition $\lambda$ there exists an FI-module $P(\lambda)$, obtained as a sub-FI-module of the free FI-module $M(\lambda)$, satisfying
$$P(\lambda)(\textbf{s}) =  \left\{\begin{array}{ccc}P(\lambda)_s &  & s \geq |\lambda|+\lambda_1 \\0 &  & \mbox{else}\end{array}\right.$$
$P(\lambda)$ has weight $\leq |\lambda|$ and stability degree $\leq \lambda_1$.
\end{proposition}

\subsection{Homological techniques for FI-modules}
One advantage of the FI-modules viewpoint is that it brings homological techniques to bear. In particular, the following result governs the dynamics of stability type through our spectral sequence. Recall that $\mathbbm{k}$ has characteristic 0.

\begin{proposition}[\cite{CEF}, Lemma 6.3.2]\label{CEF2.44}\label{CEF2.45}~\\
\begin{enumerate}
\item Let $U$, $V$, $W$ be FI-modules with stability type $(\ast, A), (B, C), (D, \ast)$ respectively, and let
\begin{equation}
U \xrightarrow{f} V \xrightarrow{g} W \notag
\end{equation}
be a complex of FI-modules (i.e., $g\circ f =0 $). Then $\ker g / \im f$ has injectivity degree $\leq \max(A,B)$ and surjectivity degree $\leq \max(C,D)$.

\item Let $V$ be an FI-module with a filtration
\begin{equation}
0 = F_jV \subseteq F_{j-i}V \subseteq \cdots \subseteq F_1V \subseteq F_0V=V \notag
\end{equation}
by FI-modules $F_iV$. The successive quotients $F_iV/F_{i+1}V$ have stability type $(\mathcal{I,S})$ for all $i$ if and only if $V$ has stability type $(\mathcal{I,S})$.
\end{enumerate}
\end{proposition}

\section{The groups $\Gamma_{n,s}$ and the FI-modules they determine} \label{Sec:Groups}
Let $X_{n,s}$ be a graph of rank $n$ with $s$ marked points $\partial := \{x_1, \ldots, x_s\}$. We defined $\Gamma_{n,s}$ as the space self homotopy equivalences of $X_{n,s}$ fixing $\partial$ (pointwise) modulo homotopies that fix $\partial$ throughout. The group operation on $\Gamma_{n,s}$ is induced by composition of homotopy equivalences, which is clearly associative and admits an identity element. In \cite{CHKV1} they prove the existence of inverse as follows. Let $f:X\rightarrow Y$ be a homotopy equivalence of graphs that sends $\partial_X=\{x_1, \ldots, x_s\} \subset X$ bijectively to $\partial_Y= \{y_1, \ldots, y_s\} \subset Y$. Consider the mapping cylinder of $f$, or rather its quotient $Z$ obtained by collapsing the $s$ intervals $x_i \times I$. By observing that the inclusion of $Y$ into $Z$ is a homotopy equivalence, and that $Z$ deformation retracts onto $X$ we obtain an inverse to $f$ that acts as $f^{-1}$ on $\partial_Y$ as desired. Moreover, this argument shows that $\Gamma_{n,s}$ does not depend on the choice of graph $X_{n,s}$ up to isomorphism.\\

The proof of Theorem B in the case when $n>1$ relies on a spectral sequence argument that itself is borne of certain short exact sequences we describe now (for full details see \cite{CHKV1}, Section 1.2).\\

Let $n >1$, $s \geq 0$ and write $X = X_{n,s}$. Let $E$ denote the space of homotopy equivalences of $X$ with no requirement that $\partial$ be fixed, and let $D$ be the space of homotopy equivalences of $X$ that are required to fix $\partial$. Thus $\Gamma_{n,s} \cong \pi_0(D)$ and $\Gamma_{n,0}=\Out(F_n) \cong \pi_0(E)$. There is a map $E \rightarrow X^s$ by $f \mapsto (f(x_1), \ldots, f(x_s))$ which is a fibration with fiber $D$ over the point $(x_1, \ldots, x_s)$. The long exact sequence of homotopy groups of this fibration ends,
\begin{equation}
\pi_1(E) \rightarrow \pi_1(X^s) \rightarrow \Gamma_{n,s} \rightarrow \Out(F_n) \rightarrow 1. \label{LES}
\end{equation}
To see that $\pi_1(E)=1$ we consider the case $s=1$, when (\ref{LES}) says,
$$\pi_1(E)\rightarrow F_n \rightarrow \Aut(F_n) \rightarrow \Out(F_n) \rightarrow 1.$$
The map $F_n \rightarrow \Aut(F_n)$ can be seen to be conjugation, and it's kernel, $\pi_1(E)$, is thus trivial. Thus we have proved the following.

\begin{proposition}[see \cite{CHKV1}, Proposition 1.2] \label{Prop:SES} If $n>1$ there is a short exact sequence,
$$1 \rightarrow F_n^s \rightarrow \Gamma_{n,s} \rightarrow \Out(F_n) \rightarrow 1$$
\end{proposition}

\subsection{Building the FI-module $H^i(\Gamma_{n, \bullet})$} The group cohomology $H^i(\Gamma_{n,s})$ admits an action of the symmetric group $\mathfrak{S}_s$, and thus defines an FB-module $H^i(\Gamma_{n,\bullet})$ taking the finite set $\textbf{s}$ to $H^i(\Gamma_{n,s})$. We now show that $H^i(\Gamma_{n,\bullet})$ actually determines an FI-module.

\begin{proposition} \label{Prop:FImodule}
Fix $i, n \geq 0$. $H^i(\Gamma_{n,\bullet})$ is an FI-module.
\end{proposition}

\begin{proof}
It suffices to describe a functorial way to assign to an injection $\phi \in \Hom_\FI(\textbf{t}, \textbf{s})$ a linear map $H^i(\Gamma_{n,t}) \rightarrow H^i(\Gamma_{n,s})$. Fix a graph $X_{n,s}$ obtained by attaching $s$ hairs to the rose $R_n$ at its single vertex. The marked points are the 1-valent vertices of $X_{n,s}$, which we identify with $\textbf{s}$. Define $X_{n,t}$ similarly and identify its marked points with $\textbf{t}$. Pick a homotopy equivalence $f:X_{n,t} \rightarrow X_{n,s}$ that acts as $\phi$ on the marked points of $X_{n,t}$; that is, the marked point $x$ in $X_{n,t}$ should be sent to the marked points $\phi(x)$ in $X_{n,s}$. Let $g$ be a self homotopy equivalence of $X_{n,s}$ fixing the hairs so that $g$ determines an element of $\Gamma_{n,s}$. Now the conjugate $fgf^{-1}$ is a self homotopy equivalence of $X_{n,t}$ fixing its marked points $\textbf{t}$ and as such determines an element $h$ of $\Gamma_{n,t}$.\\

This procedure determines a map $\Gamma_{n,s} \rightarrow \Gamma_{n,t}$ that depends on the choice of $f$. However, up to conjugation by $\Gamma_{n,t}$ this element $h$ only depends on $\phi$, and thus induces a well-defined map on cohomology, which doesn't see inner automorphisms.
\end{proof}

\begin{remark}
If we choose $f$ that it only permutes the hairs of $X_{n,t}$ (i.e., induces the identity on $\pi_1$) then the map $\Gamma_{n,s} \rightarrow \Gamma_{n,t}$ can be thought of as \emph{forgetting} the $s-t$ points not in the image of $\phi$, and relabelling the hairs according to $\phi$.
\end{remark}

\begin{remark}
It is perhaps tempting to draw on the $\FI$ structure at the level of groups and try and make $\Gamma_{1,\bullet}$ a (contravarient) functor from $\FI$ to the category of groups; a co-FI-group in the language of \cite{CEF}. Indeed, to an injection $\phi \in \Hom_\FI(t,s)$ we described maps from $\Gamma_{n,s}$ to $\Gamma_{n,t}$ in the proof of Proposition \ref{Prop:FImodule}. However, the element $fgf^{-1}$ in $\Gamma_{n,t}$ depended on the choice of homotopy equivalence $f$ and as such it is false that $\Gamma_{n,\bullet}$ forms a co-FI-group in general. That being said, it is straightforward to show that $\Gamma_{1,\bullet}$ does form a co-FI-group. We have no need for that result here, and therefore don't do so.
\end{remark}

A useful observation is that the structure maps are always injective. 

\begin{proposition} \label{Prop:Injective}
The structure maps of $H^i(\Gamma_{n,\bullet})$ are injective.
\end{proposition}

\begin{proof}
It suffices to show that the maps $\Gamma_{n,s}\rightarrow \Gamma_{n,t}$ used to build the structure maps are split. In the case $t \neq0$ this is a straightforward adaptation of a result in \cite{CHKV1} (see Proposition 1.2) where they prove that there is a splitting $\Gamma_{n,s} \rightarrow \Gamma_{n,s-k}$ when $k<s$. The remaining case where $t=0$ is also dealt with in \cite{CHKV1}, (see Theorem 1.4) where they prove that natural map $\Aut(F_n) \rightarrow \Out(F_n)$ splits on the level of (rational) (co)homology.
\end{proof}

\subsection{Rank 1} In rank 1 the situation is somewhat simpler and we don't need to appeal to a spectral sequence argument to witness Theorem B.

\begin{proposition} As FI-modules
\begin{equation}
H^i(\Gamma_{1,\bullet}) = \left\{\begin{array}{ccc}P(1^i) &  & i~\mbox{even} \\0 &  & i~\mbox{odd}\end{array}\right. \notag
\end{equation}
and as such $H^i(\Gamma_{1,\bullet})$ satisfies Theorem B.
\end{proposition}

\begin{proof}

The $\mathfrak{S}_s$-module structure of the cohomology in rank 1 was computed in \cite{CHKV1}, Proposition 2.7 to be
\begin{equation}
H^i(\Gamma_{1,s}) = \left\{\begin{array}{ccc}P_{(s-i,1^i)} &  & i~\mbox{even} \\0 &  & i~\mbox{odd}\end{array}\right. \notag
\end{equation}
We need only consider the even case, where we have an FI-module which behaves like $P(1^i)$ when evaluated at any finite set. The fact that the irreducible decomposition contains exactly one irreducible at each finite set implies that the structure maps either agree with those of $P(1^i)$ or are zero. Proposition \ref{Prop:Injective} says that the structure maps are injective, so we have an equality of FI-modules $H^i(\Gamma_{1,\bullet})=P(1^i)$ when $i$ is even. As for satisfying Theorem B, $P(1^i)$ has stability degree $\leq1$ and weight $\leq i$ by Proposition \ref{Prop:Plambda}, as desired.
\end{proof}

\section{The spectral sequence argument} \label{SecSSA} \label{SecGG}
With the rank 1 case taken care of we proceed to prove Theorem B in higher rank by establishing a spectral sequence of FI-modules converging to $H^i(\Gamma_{n,\bullet})$. Throughout this section fix $i \geq 0$ and $n \geq 2$.

\begin{lemma} \label{Lem:SS}There is a spectral sequence of FI-modules
\begin{equation}
E_2^{pq} = \bigoplus_{|\lambda|=q, \lambda_1\leq n} C_{p, \lambda} \otimes M(\lambda) \Rightarrow_p H^i(\Gamma_{n,\bullet}), \notag
\end{equation}
converging to the FI-module $H^i(\Gamma_{n,\bullet})$, where $C_{p, \lambda}$ is a constant FI-module depending only on $p$ and $\lambda$.
\end{lemma}

\begin{proof} Let $$C_{p, \lambda}=H^p(\Out(F_n); \mathbb{S}_{\lambda'} H)$$ where $H:= H^1(F_n)=\mathbbm{k}^n$ and where $\mathbb{S}_{\lambda'}$ is the Schur functor corresponding to the conjugate partition $\lambda'$. Define $E_2^{pq}$ as in the statement of the lemma.\\

We will show that, when evaluated at the finite set $\mathbf{s}=\{1, \ldots, s\}$, this gives the second page of the Leray-Serre spectral sequence of groups associated to the short exact sequence 
$$1 \rightarrow F_n^s \rightarrow \Gamma_{n,s} \rightarrow \Out(F_n) \rightarrow 1$$
from Proposition \ref{Prop:SES}. In other words, we will show that
\begin{equation}
(E_2^{pq})_s = H^p(\Out(F_n);H^q(F_n^s)) \Rightarrow_p H^{p+q}(\Gamma_{n,s}) \notag
\end{equation}
as a spectral sequence of groups. Functoriality of the Leray-Serre spectral sequence will complete the proof.\\

First observe that, by the K\"unneth formula, $H^q(F_n^s)=H^{\wedge q} \circ P_{(s-q)}$ as an $\mathfrak{S}_s$-module (this is proved carefully in \cite{CHKV1}, Lemma 2.4). We have,
\begin{equation}
H^p(\Out(F_n); H^q(F_n^s)) = H^p(\Out(F_n);H^{\wedge q}\circ P_{(s-q)}). \notag
\end{equation}
The $\Out(F_n)$ action on $H^{\wedge q}$ factors through a $GL_n(\mathbb{Z})$ action. We decompose using Schur-Weyl duality giving,
\begin{align}
H^p(\Out(F_n);H^{\wedge q}\circ P_{(s-q)}) &= \bigoplus_{|\lambda| = q} H^p(\Out(F_n); \mathbb{S}_{\lambda}H \otimes P_{\lambda'} \circ P_{(s-q)} ) \notag \\
&= \bigoplus_{|\lambda| = q} H^p(\Out(F_n); \mathbb{S}_{\lambda}H) \otimes P_{\lambda'} \circ P_{(s-q)} \notag \\
&= \bigoplus_{|\lambda| = q} H^p(\Out(F_n); \mathbb{S}_{\lambda}H) \otimes M(\lambda')_s \notag
\end{align}
where $\lambda'$ is the conjugate partition of $\lambda$. Now observe that $\mathbb{S}_\lambda H=0$ if $\lambda$ has more than $n$ rows by the character formula (for details see \cite{FH}). Therefore ${\lambda'}$ has at most $n$ columns, i.e., $\lambda'_1 \leq n$. Therefore
\begin{equation}
H^p(\Out(F_n); H^q(F_n^s)) = \bigoplus_{\substack{|\lambda'|=q \\ \lambda'_1\leq n}} H^p(\Out(F_n); \mathbb{S}_{\lambda}H) \otimes M(\lambda')_s = (E_2^{pq})_s .\notag
\end{equation}
Swapping $\lambda$ with $\lambda'$ completes the proof.
\end{proof}

We now describe the stability type and weight of the FI-module $E_2^{pq}$.

\begin{lemma} \label{E2StabW}
$E_2^{pq}$ has stability type $(0,n)$ and weight $q$.
\end{lemma}

\begin{proof}
$C_{p, \lambda}$ are constant FI-modules and thus do not contribute to weight or stability type. $M(\lambda)$ has stability type $(0,\lambda_1)$ and weight $|\lambda|$ by Proposition \ref{stabM}. $E_2^{pq}$ is obtained by summing over partitions $\lambda \vdash q$ with at most $n$ columns. In particular, each $\lambda$ satisfies $\lambda_1 \leq n$ and $|\lambda|=q$.
\end{proof}

We are ready to give the spectral sequence argument.

\begin{lemma} \label{Spec}
The FI-modules $E_k^{pq}$ on the $k$th page of the spectral sequence have stability degree $n$.
\end{lemma}

\begin{proof}
We denote the stability type of $E_k^{pq}$ by $(\mathcal{I}^{pq}_k, \mathcal{S}_k^{pq})$. We use Lemma \ref{CEF2.44} and the fact that the spectral sequence is concentrated in the first quadrant to inductively produce bounds on stability type in subsequent pages.\\

On the 2nd page all terms $E_2^{pq}$ have stability type at most $(0, n)$. To compute the terms on the third page we use the differentials $d_2^{pq}$ of bidegree $(2,-1)$. We indicate the stability type of terms for convenience.
\begin{center}
\begin{tikzpicture}[scale=1]
\node at (-0.2,0) {$E_2^{p-2, q+1}$}; \draw[->] (0.5,0) -- (2.5,0); \node at (1.5, 0.3) {\small{$d_2^{p+2, q-1}$}};
\node at (3,0) {$E_2^{pq}$}; \draw[->] (3.5,0) -- (5.1,0); \node at (4.3, 0.3) {\small{$d_2^{pq}$}};
\node at (6,0) {$E_2^{p+2, q-1}$};
\node at (-0.2,-0.7) {\small{$(0,\mathcal{S}^{p-2, q+1}_2)$}};
\node at (3,-0.7) {\small{$(0,n)$}};
\node at (6,-0.7) {\small{$(0,n)$}};
\end{tikzpicture}
\end{center}
where
\begin{equation}
\mathcal{S}^{p-2,q+1}_2 = \left\{\begin{array}{ccc}0 & ~ & p=0,1 \\n & ~ & p \geq 2\end{array}\right. \notag
\end{equation}
depending on whether or not $E_2^{p-2,q+1}$ is in the first quadrant. Now Lemma \ref{CEF2.44} gives us that
\begin{align}
\mathcal{I}_3^{pq} &= \max(\mathcal{S}_2^{p-2, q+1}, \mathcal{I}_2^{pq})= \left\{\begin{array}{ccc}0 & ~ & p=0,1 \\n & ~ & p \geq 2\end{array}\right. \notag \\
\mathcal{S}_3^{pq} &= \max(\mathcal{S}_2^{pq}, \mathcal{I}_2^{p+2, q-1}) = n. \notag
\end{align}
We proceed similarly with the inductive step. We have
\begin{center}
\begin{tikzpicture}[scale=1]
\node at (-0.5,0) {$E_k^{p-k, q+k-1}$}; \draw[->] (0.5,0) -- (2.5,0); \node at (1.5, 0.3) {\small{$d_k^{p+k, q-k+1}$}};
\node at (3,0) {$E_k^{pq}$}; \draw[->] (3.5,0) -- (5.1,0); \node at (4.3, 0.3) {\small{$d_k^{pq}$}};
\node at (6.2,0) {$E_k^{p+k, q-k+1}$};
\node at (-0.4,-0.7) {\small{$(\ast,\mathcal{S}_k^{p-k,q+k-1})$}};
\node at (3,-0.7) {\small{$(\mathcal{I}_k^{pq},n)$}};
\node at (6,-0.7) {\small{$(n,\ast)$}};
\end{tikzpicture}
\end{center}
and can finally conclude that
\begin{align}
\mathcal{I}_{k+1}^{pq} &= \max(\mathcal{S}_k^{p-k,q+k-1}, \mathcal{I}_k^{pq})= \left\{\begin{array}{ccc}0 & ~ & p=0,1 \\n & ~ & p \geq 2\end{array}\right. \notag \\
\mathcal{S}_{k+1}^{pq} &= \max(\mathcal{S}_k^{pq}, \mathcal{I}_k^{p+k, q-k+1}) = n, \notag
\end{align}
since $\mathcal{S}_k^{p-k,q+k-1}=n$ unless $p<k$ when it is 0, and $\mathcal{I}_k^{pq}=n$ unless $p=0,1$ when it is 0.
\end{proof}

We are now ready to prove our main result.

\begin{proof}[Proof of Theorem B]
First observe that $E_\infty^{p, i-p}$ has weight $\leq i$ and stability degree $\leq n$. Indeed $E_\infty^{p, i-p}$ is a subquotient of $E_2^{p, i-p}$ and thus has weight $\leq i-p \leq i$ by Lemma \ref{E2StabW}, and $E_\infty^{p, i-p}$ has stability degree $n$ by Lemma \ref{Spec}.\\

We have that $E_2^{pq}$ is a first quadrant spectral sequence of FI-modules converging to the FI-module $H^{p+q}(\Gamma_{n,\bullet})$. This tells us that there exists a natural filtration of $H^i(\Gamma_{n,\bullet})$ whose graded quotients are $E_\infty^{p, i-p}$ and so, by Proposition \ref{CEF2.45}, $H^i(\Gamma_{n,\bullet})$ has stability degree $n$. Moreover, since weight is preserved under extensions, $H^i(\Gamma_{n,\bullet})$ has weight $\leq i$.
\end{proof}

\begin{remark}
In \cite{JR} Jim\'enez Rolland develops a general framework for dealing with spectral sequences of FI-modules. In particular, the description given in the proof of Lemma \ref{Lem:SS} shows that $H^q(F_n^s)$ has weight and stability degree $\leq q$. Thus, in the notation of \cite{JR} Theorem 5.3 we have shown that $\beta = 1$ and that $H^i(\Gamma_{n,\bullet})$ has weight $\leq i$ and stability type $(2i,i)$. We note that this recovers the representation stability bounds of \cite{CHKV1} upon which we just improved.
\end{remark}

\bibliographystyle{acm}

\end{document}